\documentclass{article}%
\usepackage{amsmath}
\usepackage{amsfonts}
\usepackage{amssymb}
\usepackage{graphicx}%
\newtheorem{theorem}{Theorem}[section]

\newtheorem{corollary}[theorem]{Corollary}

\newtheorem{definition}[theorem]{Definition}
\newtheorem{example}[theorem]{Example}

\newtheorem{lemma}[theorem]{Lemma}
\newtheorem{notation}[theorem]{Notation}

\newtheorem{proposition}[theorem]{Proposition}
\newtheorem{remark}[theorem]{Remark}

\newenvironment{proof}[1][Proof]{\noindent\textbf{#1.} }{\ \rule{0.5em}{0.5em}}
\begin{document}

\title{Notes on Simple Modules over Leavitt Path Algebras}
\author{Kulumani M. Rangaswamy\\University of Colorado at Colorado Springs\\Colorado Springs, Colorado 80918.\\E-mail: krangasw@uccs.edu}
\date{}
\maketitle

\begin{abstract}
Given an arbitrary graph $E$ and any field $K$, a new class of simple modules
over the Leavitt path algebra $L_{K}(E)$ is constructed by using vertices that
emit infinitely many edges in $E$. The corresponding annihilating primitive
ideals are also described. Using a Boolean subring of idempotents, bounds for
the cardinality of the set of distinct isomorphism classes of simple
$L_{K}(E)$-modules are given. We also append other information about the
algebra $L_{K}(E)$ of a finite graph $E$ over which every simple module is
finitely presented.

\end{abstract}

\section{Introduction and Preliminaries}

Leavitt path algebras were introduced in \cite{AA}, \cite{AMP} as algebraic
analogues of graph C*-algebras and as natural generalizations of Leavitt
algebras of type (1,n) built in \cite{L}. The various ring-theoretical
properties of these algebras have been actively investigated in a series of
papers (see, for e.g., \cite{AA}, \cite{AR}, \cite{AAJZ-1}, \cite{AMP},
\cite{ARS1}, \cite{R1}, \cite{R2}, \cite{T}). In contrast, the module theory
of Leavitt path algebras $L_{K}(E)$ of $\sqrt{}$arbitrary directed graphs $E$
over a field $K$ is still at its infancy. The initial organized attempt to
study $L_{K}(E)$-modules was done in \cite{AB} where, for a finite graph $E$,
the simply presented $L_{K}(E)$-modules were described in terms of finite
dimensional representations of the usual path algebras of the reverse graph
$\bar{E}$ of $E$. As an important step in the study of modules over a Leavitt
path algebra $L_{K}(E)$, the investigation of the simple $L_{K}(E)$-modules
has recently received some attention (see \cite{AR1}, \cite{AR2}, \cite{AMMS},
\cite{C}). Following the ideas of Smith \cite{S1}, Chen \cite{C} constructed
irreducible representations of $L_{K}(E)$ by using sinks and tail-equivalent
classes of infinite paths in the graph $E$. Chen's construction was expanded
in \cite{AR1} to introduce additional classes of non-isomorphic simple
$L_{K}(E)$-modules. In section 2 of this paper, we construct a new class of
simple left $L_{K}(E)$-modules induced by vertices which are infinite emitters
and at the same time streamline the process of construction of certain simple
modules introduced in \cite{AR1}. A description of the annihilating primitive
ideals of these simple modules shows that these new simple modules\ are
distinct from (i.e., not isomorphic to) any of the previously constructed
simple $L_{K}(E)$-modules in \cite{AR1}, \cite{AMMS} and \cite{C}. In section
3, we adapt the ideas of Rosenberg \cite{RO} to show that the cardinality of
any single isomorphism class of simple left $L_{K}(E)$-modules has at most the
cardinality of $L_{K}(E)$. Using a Boolean subring of commuting idempotents
induced by the paths in $L_{K}(E)$, we obtain a lower bound for the
cardinality of the set of non-isomorphic simple $L_{K}(E)$-modules. In
particular, if $L_{K}(E)$ is a countable dimensional simple algebra, then it
will have exactly $1$ or at least $2^{\aleph_{0}}$ distinct isomorphism
classes of simple modules. In section 4, we include some improvements and
simplification of the results of \cite{AAJZ-2} dealing with the structure of
Leavitt path algebras over which every simple module is finitely presented.

For the general notation, terminology and results in Leavitt path algebras, we
refer to \cite{AA}, \cite{AAS}, \cite{AMP}. We give below a short outline of
some of the needed basic concepts and results.

A (directed) graph $E=(E^{0},E^{1},r,s)$ consists of two sets $E^{0}$ and
$E^{1}$ together with maps $r,s:E^{1}\rightarrow E^{0}$. The elements of
$E^{0}$ are called \textit{vertices} and the elements of $E^{1}$
\textit{edges}. All the graphs $E$ that we consider (excepting those studied
in section 4) are arbitrary in the sense that no restriction is placed either
on the number of vertices in $E$ or on the number of edges emitted by a single
vertex. Also $K$ stands for an arbitrary field.

A vertex $v$ is called a \textit{sink} if it emits no edges and a vertex $v$
is called a \textit{regular} \textit{vertex} if it emits a non-empty finite
set of edges. An \textit{infinite emitter} is a vertex which emits infinitely
many edges. For each $e\in E^{1}$, we call $e^{\ast}$ a ghost edge. We let
$r(e^{\ast})$ denote $s(e)$, and we let $s(e^{\ast})$ denote $r(e)$. A\textit{
path} $\mu$ of length $n>0$ is a finite sequence of edges $\mu=e_{1}e_{2}%
\cdot\cdot\cdot e_{n}$ with $r(e_{i})=s(e_{i+1})$ for all $i=1,\cdot\cdot
\cdot,n-1$. In this case $\mu^{\ast}=e_{n}^{\ast}\cdot\cdot\cdot e_{2}^{\ast
}e_{1}^{\ast}$ is the corresponding ghost path. Any vertex is considered a
path of length $0$. The set of all vertices on the path $\mu$ is denoted by
$\mu^{0}$.

A path $\mu$ $=e_{1}\dots e_{n}$ in $E$ is \textit{closed} if $r(e_{n}%
)=s(e_{1})$, in which case $\mu$ is said to be based at the vertex $s(e_{1})$.
A closed path $\mu$ as above is called \textit{simple} provided it does not
pass through its base more than once, i.e., $s(e_{i})\neq s(e_{1})$ for all
$i=2,...,n$. The closed path $\mu$ is called a \textit{cycle} if it does not
pass through any of its vertices twice, that is, if $s(e_{i})\neq s(e_{j})$
for every $i\neq j$.

If there is a path from vertex $u$ to a vertex $v$, we write $u\geq v$. A
subset $D$ of vertices is said to be \textit{downward directed }\ if for any
$u,v\in D$, there exists a $w\in D$ such that $u\geq w$ and $v\geq w$. A
subset $H$ of $E^{0}$ is called \textit{hereditary} if, whenever $v\in H$ and
$w\in E^{0}$ satisfy $v\geq w$, then $w\in H$. A hereditary set is
\textit{saturated} if, for any regular vertex $v$, $r(s^{-1}(v))\subseteq H$
implies $v\in H$.

Given an arbitrary graph $E$ and a field $K$, the \textit{Leavitt path algebra
}$L_{K}(E)$ is defined to be the $K$-algebra generated by a set $\{v:v\in
E^{0}\}$ of pairwise orthogonal idempotents together with a set of variables
$\{e,e^{\ast}:e\in E^{1}\}$ which satisfy the following conditions:

(1) \ $s(e)e=e=er(e)$ for all $e\in E^{1}$.

(2) $r(e)e^{\ast}=e^{\ast}=e^{\ast}s(e)$\ for all $e\in E^{1}$.

(3) (The "CK-1 relations") For all $e,f\in E^{1}$, $e^{\ast}e=r(e)$ and
$e^{\ast}f=0$ if $e\neq f$.

(4) (The "CK-2 relations") For every regular vertex $v\in E^{0}$,
\[
v=\sum_{e\in E^{1},s(e)=v}ee^{\ast}.
\]

For any vertex $v$, the \textit{tree} of $v$ is $T(v)=\{w:v\geq w\}$. We say
there is a bifurcation at a vertex $v$, if $v$ emits more than one edge. In a
graph $E$, a vertex $v$ is called a \textit{line point} if there is no
bifurcation or a cycle based at any vertex in $T(v)$. Thus, if $v$ is a line
point, there will be a single finite or infinite line segment $\mu$ starting
at $v$ ($\mu$ could just be $v$) and any other path $\alpha$ with
$s(\alpha)=v$ will just be an initial sub-segment of $\mu$. It was shown in
\cite{AMMS} that $v$ is a line point in $E$ if and only if $vL_{K}(E)$ (and
likewise $L_{K}(E)v$) is a simple left (right) ideal.

We shall be using the following concepts and results from \cite{T}. A
\textit{breaking vertex }of a hereditary saturated subset $H$ is an infinite
emitter $w\in E^{0}\backslash H$ with the property that $1\leq|s^{-1}(w)\cap
r^{-1}(E^{0}\backslash H)|<\infty$. The set of all breaking vertices of $H$ is
denoted by $B_{H}$. For any $v\in B_{H}$, $v^{H}$ denotes the element
$v-\sum_{s(e)=v,r(e)\notin H}ee^{\ast}$. Given a hereditary saturated subset
$H$ and a subset $S\subseteq B_{H}$, $(H,S)$ is called an \textit{admissible
pair.} Given an admissible pair $(H,S)$, the ideal generated by $H\cup
\{v^{H}:v\in S\}$ is denoted by $I_{(H,S)}$. It was shown in \cite{T} that the
graded ideals of $L_{K}(E)$ are precisely the ideals of the form $I_{(H,S)}$
for some admissibile pair $(H,S)$. Moreover, $L_{K}(E)/I_{(H,S)}\cong
L_{K}(E\backslash(H,S))$. Here $E\backslash(H,S)$ is the \textit{Quotient
graph of }$E$ in which\textit{ }$(E\backslash(H,S))^{0}=(E^{0}\backslash
H)\cup\{v^{\prime}:v\in B_{H}\backslash S\}$ and $(E\backslash(H,S))^{1}%
=\{e\in E^{1}:r(e)\notin H\}\cup\{e^{\prime}:e\in E^{1},r(e)\in B_{H}%
\backslash S\}$ and $r,s$ are extended to $(E\backslash(H,S))^{0}$ by setting
$s(e^{\prime})=s(e)$ and $r(e^{\prime})=r(e)^{\prime}$.

A useful observation is that every element $a$ of $L_{K}(E)$ can be written as
$a=%
{\textstyle\sum\limits_{i=1}^{n}}
k_{i}\alpha_{i}\beta_{i}^{\ast}$, where $k_{i}\in K$, $\alpha_{i},\beta_{i}$
are paths in $E$ and $n$ is a suitable integer. Moreover, $L_{K}%
(E)=\oplus_{v\in E^{0}}L_{K}(E)v=\oplus_{v\in E^{0}}vL_{K}(E)$ (see \cite{AA}).

Even though the Leavitt path algebra $L_{K}(E)$ may not have the
multiplicative identity $1$, we shall write $L_{K}(E)(1-v)$ to denote the set
$\{x-xv:x\in L_{K}(E)\}$. If $v$ is an idempotent or a vertex, we get a direct
decomposition $L_{K}(E)=L_{K}(E)v\oplus L_{K}(E)(1-v)$.

\section{A new class of simple modules}

Let $E$ be an arbitrary graph. Throughout this section, we shall use the following notation.

For $v\in E^{0}$ define
\[
M(v)=\{w\in E^{0}:w\geq v\}\text{ and }H(v)=E^{0}\setminus M(v)=\{u\in
E^{0}:u\ngeqslant v\}.
\]

Clearly $M(v)$ is downward directed. Also, for any vertex $v$ which is a sink
or infinite emitter, the set $H(v)$ is a hereditary saturated subset of $E$.
If $v$ is a finite emitter, it might be that $H(v)$ is not saturated, and that
$v$ belongs to the saturation of $H(v)$.

For convenience in writing, we shall denote the Leavitt path algebra
$L_{K}(E)$ by $L$.

\begin{definition}
\textbf{The }$L$\textbf{-module }$S_{v\infty}:$

Suppose $v$ is an infinite emitter in $E$. Define\textbf{ }$\mathbf{S}%
_{v\infty}$ to be the $K$-vector space having as a basis the set $B=\{p:p$ a
path in $E$ with $r(p)=v\}$. Following Chen \cite{C}, we define, for each
vertex $u$ and each edge $e$ in $E$, linear transformations $P_{u},S_{e}$ and
$S_{e^{\ast}}$ on $\mathbf{S}_{v\infty}$ as follows:

For all paths $p\in B$,

$P_{u}(p)=\left\{
\begin{array}
[c]{c}%
p\text{, if }u=s(p)\\
0\text{, otherwise}%
\end{array}
\right.  $

$S_{e}(p)=\left\{
\begin{array}
[c]{c}%
ep\text{, if }r(e)=s(p)\\
0\text{, \qquad otherwise}%
\end{array}
\right.  $

$S_{e^{\ast}}(v)=0$

$S_{e^{\ast}}(p)=\left\{
\begin{array}
[c]{c}%
p^{\prime}\text{, if }p=ep^{\prime}\\
0\text{, otherwise}%
\end{array}
\right.  $
\end{definition}

Then it can be checked that the endomorphisms $\{P_{u},S_{e},S_{e^{\ast}}:u\in
E^{0},e\in E^{1}\}$ satisfy the defining relations (1) - (4) of the Leavitt
path algebra $L$. This induces an algebra homomorphism $\phi$ from $L$ to
$End_{K}(\mathbf{S}_{v\infty})$ mapping $u$ to $P_{u}$, $e$ to $S_{e}$ and
$e^{\ast}$to $S_{e^{\ast}}$. Then $\mathbf{S}_{v\infty}$ can be made a left
module over $L$ via the homomorphism $\phi$. We denote this $L$-module
operation on $\mathbf{S}_{v\infty}$ by $\cdot$.

\begin{remark}
The above construction does not work if $v$ is a regular vertex. Specifically,
the needed CK-2 relation $%
{\textstyle\sum\limits_{e\in s^{-1}(v)}}
S_{e}S_{e^{\ast}}=P_{v}$ does not hold. Because, on the one hand $(%
{\textstyle\sum\limits_{e\in s^{-1}(v)}}
S_{e}S_{e^{\ast}})(v)=0$ but on the other hand $P_{v}(v)=v\neq0$.
\end{remark}

\begin{proposition}
\label{InfiniteEmitterSimple}For each infinite emitter $v$ in $E$,
$\mathbf{S}_{v\infty}$ is a simple left module over $L_{K}(E)$.
\end{proposition}

\begin{proof}
\bigskip\ Suppose $U$ is non-zero submodule of $\mathbf{S}_{v\infty}$ and let
\[
0\neq a=%
{\textstyle\sum\limits_{i=1}^{n}}
k_{i}p_{i}\in U\qquad\qquad\qquad\qquad(\#)
\]
where $k_{i}\in K$ and the $p_{i}$ are paths in $E$ with $r(p_{i})=v$ and we
assume that the paths $p_{i}$ are all different.

By induction on $n$, we wish to show that $v\in U$. Suppose $n=1$ so that
$a=k_{1}p_{1}$. Then $p_{1}^{\ast}\cdot a=k_{1}v\in U$ and we are done.
Suppose $n>1$ and assume that $v\in U$ if $U$ contains a non-zero element
which is a $K$-linear combination of less than $n$ paths. Among the paths
$p_{i}$, assume that $p_{1}$ has the smallest length. If $p_{1}$ has length
$0$, that is, if $p_{1}=v$ and, for some $s$, $p_{s}$ is a path of length
$>0$, then since $p_{s}^{\ast}\cdot v=0$, $p_{s}^{\ast}\cdot a\in U$ will be a
sum of less than $n$ terms and so by induction $v\in U$. Suppose $p_{1}$ has
length $>0$. Now $p_{1}^{\ast}\cdot a=a^{\prime}\in U$. If $p_{1}^{\ast}\cdot
a$ is a sum of less than $n$ terms, we are done. Otherwise, $a^{\prime}%
=k_{1}p_{1}^{\ast}\cdot p_{1}+b=k_{1}v+b$ where $b\ $is a sum of less than $n$
terms with its first non-zero term, say, $k_{t}p_{t}$. Then $p_{t}^{\ast}\cdot
a^{\prime}=p_{t}^{\ast}\cdot b\in U$ and $0\neq p_{t}^{\ast}\cdot b$ is a sum
of less than $n$ terms. Hence by induction, $v\in U$ and we conclude that
$U=\mathbf{S}_{v\infty}$.
\end{proof}

The next proposition describes the annihilating primitive ideal of the simple
module $\mathbf{S}_{v\infty}$.

\begin{proposition}
\label{Annihilator Ideal}Let $v$ be an infinite emitter. Then
\end{proposition}

$Ann_{L_{K}(E)}(\mathbf{S}_{v\infty})=\left\{
\begin{array}
[c]{c}%
I(H(v),B_{H(v)})\text{, \ \ \ \ \ \ \ \ \ \ \ \ \ \ \ \ \ \ \ \ if }%
|s^{-1}(v)\cap r^{-1}(M(v))|=0\text{;}\\
I(H(v),B_{H(v)\backslash\{v\}})\text{, \quad if }|s^{-1}(v)\cap r^{-1}%
(M(v))|\neq0\text{ and finite}\\
I(H(v),B_{H(v)})\text{, \qquad\ \ if }|s^{-1}(v)\cap r^{-1}(M(v))|\text{ is
infinite.\quad\ }%
\end{array}
\right.  $

\begin{proof}
Let $J=Ann_{L_{K}(E)}(\mathbf{S}_{v\infty})\ $. Clearly $H(v)\subset J$ since
for any $u\in H(v)$, $u\ngeqslant v$ and so $u\cdot p=0$ for all $p$ with
$r(p)=v$. Indeed $J\cap E^{0}=H(v)$.

Suppose $|s^{-1}(v)\cap r^{-1}(M(v))|=0$ so that $r(s^{-1}(v))\subseteqq
H(v)$. Let $u\in B_{H(v)}$. Clearly $u^{H(v)}\cdot p=0$ if $p$ is a path with
$r(p)=v$ and $s(p)\neq u$. On the other hand, if $p$ is a path from $u$ to $v$
with $p=ep^{\prime}$ where $e$ is an edge, then%
\[
u^{H(v)}\cdot p=(u-%
{\textstyle\sum\limits_{f\in s^{-1}(u),r(f)\notin H(v)}}
ff^{\ast})\cdot p=(e-e)p^{\prime}=0.
\]
This shows that $I(H(v),B_{H(v)})\subseteq J$. Now $J\cap E^{0}=H(v)$. If $J$
were a non-graded ideal, then it follows from the proof of Theorem 3.12 (iii)
of \cite{R1}, that $v$ will be the base of a cycle $c$ with $c^{0}\subset
M(v)$. In particular, there is an edge $e$ with $s(e)=v$ and $r(e)\in M(v)$.
But this is not possible since $r(s^{-1}(v))\subseteq H(v)$. Thus $J$ is a
graded ideal with $J\cap E^{0}=H(v)$. Since $I(H(v),B_{H(v)})$ is the largest
graded ideal for which $I(H(v),B_{H(v)})\cap E^{0}=H(v)$, we conclude that
$J=$ $I(H(v),B_{H(v)})$.

Suppose $|s^{-1}(v)\cap r^{-1}(M(v))|\neq0$ and is finite so that $v\in
B_{H(v)}$. If $u\in B_{H(v)}$ with $u\neq v$, then the arguments in the
preceding paragraph shows that $u^{H(v)}\in J$. But $v^{H(v)}\notin J$ since%
\[
v^{H(v)}\cdot v=(v-%
{\textstyle\sum\limits_{f\in s^{-1}(v),L(f)\notin H(v)}}
ff^{\ast})\cdot v=v\cdot v-0=v\neq0.
\]
This shows that $I(H(v),B_{H(v)}\backslash\{v\})\subseteq J$. Now $J$ is a
primitive ideal with $J\cap E^{0}=H(v)$. If $J$ were a non-graded ideal, then
from the description of the primitive ideals in Theorem 4.3 of \cite{R1}, we
will have $I(H(v),B_{H(v)})\subseteqq J$ and this is not possible since
$v^{H(v)}\notin J$. We then conclude that the graded ideal $J$ must be equal
to $I(H(v),B_{H(v)\backslash\{v\}})$.

Finally, suppose $|s^{-1}(v)\cap r^{-1}(M(v))|$ is infinite. This means, in
particular, there are infinitely many cycles in $M(v)$ based at $v$. As
$M(v)=E^{0}\backslash H(v)$, it \ is then clear from Theorem 3.12 of \cite{R1}
that $J$ cannot be a non-graded ideal. Also, as $v\notin B_{H(v)}$, the
earlier arguments show that $I(H(v),B_{H(v)})\subseteq J$. Observing that
$I(H(v),B_{H(v)})\cap E^{0}=H(v)=J\cap E^{0}$, we then conclude that $J=$
$I(H(v),B_{H(v)})$.
\end{proof}

Before proceeding further, we shall review the construction of some of the
simple modules introduced in \cite{C} and \cite{AR1} and refer them as simple
modules of type 1,2 or 3. In this connection, we wish to point out that the
notation and terminology used by Chen in \cite{C} is different from those used
in papers on Leavitt path algebras such as \cite{AAPS} \ while we shall follow
that of \cite{C}.

Type-1 Simple Module: Chen \cite{C} defines an equivalence relation among
infinite paths by using the following notation. If $p=e_{1}e_{2}\cdot
\cdot\cdot e_{n}\cdot\cdot\cdot$ is an infinite path where the $e_{i}$ are
edges, then for any positive integer $n$, let $\tau_{\leq n}(p)=e_{1}%
e_{2}\cdot\cdot\cdot e_{n}$ and $\tau_{>n}(p)=e_{n+1}e_{n+2}\cdot\cdot\cdot$.
Two infinite paths $p$ and $q$ are said to be \textit{tail}
\textit{equivalent}, in symbols, $p\sim q$, if there exist positive integers
$m$ and $n$ such that $\tau_{>n}(p)=\tau_{>m}(q)$. Then $\backsim$ is an
equivalence relation.

Given an equivalence class of infinite paths $[p]$, let $V_{[p]}$ denote the
$K$-vector space having the set $\{q:q\in\lbrack p]\}$ as a basis. Then Chen
\cite{C} defines an $L$-module operation on $V_{[p]}$ making $V_{[p]}$ a left
$L$-module similar to the way the module operation is defined on
$\mathbf{S}_{v\infty}$ above, except that the condition that $S_{e^{\ast}%
}(v)=0$ for any edge $e$ is dropped. Chen \cite{C} shows that the module
$V_{[p]}$ becomes a simple $L$-module.

Type-2 Simple Module: Let $w$ be a sink and $\mathbf{N}_{w}$ be a $K$-vector
space having as a basis the set $\{p:p$ paths in $E$ with $L(p)=w\}$.
Proceeding as was done above for $\mathbf{S}_{v\infty}$, Chen \cite{C} defines
an $L$-module action on $\mathbf{N}_{w}$ and shows that $\mathbf{N}_{w}$
becomes a simple module.

Type-3 Simple Modules: These additional classes of simple $L$-modules, denoted
respectively by $\mathbf{N}_{v}^{B_{H(v)}},\mathbf{N}_{v}^{H(v)}$ and
$V_{[p]}^{f}$, were introduced in \cite{AR1}:

(i) Suppose that $\ v$ is an infinite emitter such that $v\in B_{H(v)}$. Then
we can build the primitive ideal $P=I_{(H(v),B_{H(v)}\setminus\{v\})}$ (see
\cite{R1}) and the factor ring
\[
L_{K}(E)/P\cong L_{K}(F)\,
\]
where $F=E\setminus(H(v),B_{H(v)}\setminus\{v\})$. Then $F^{0}=(E^{0}\setminus
H(v))\cup\{v^{\prime}\}$,
\[
F^{1}=\{e\in E^{1}:r(e)\notin H(v)\}\cup\{e^{\prime}:e\in E^{1},r(e)=v\}
\]
and $r$ and $s$ are extended to $F$ by $s(e^{\prime})=s(e)$ and $r(e^{\prime
})=v^{\prime}$ for all $e\in E^{1}$ with $r(e)=v$. Note that $v^{\prime}$ is a
sink in $F$ and it is easy to see that $M_{F}(v^{\prime})=F^{0}$.

Accordingly, we may consider the Type 2 simple module $\mathbf{N}_{v^{\prime}%
}$ of $L_{K}(F)$ introduced by Chen corresponding to the sink $v^{\prime}$ in
$F$. Using the quotient map $L_{K}(E)\rightarrow L_{K}(F)$, we may view
$\mathbf{N}_{v^{\prime}}$ as a simple module over $L_{K}(E)$. This simple
$L_{K}(E)$-module is denoted by $\mathbf{N}_{v}^{B_{H(v)}}$.

(ii) Suppose $v$ is an infinite emitter and such that $r(s^{-1}(v))\subseteq
H(v)$. Then $v$ is the unique sink in the graph $G=E\setminus(H(v),B_{H(v)})$.
Let $\mathbf{N}_{v}$ be the corresponding Type 2 simple $L_{K}(E\setminus
(H(v),B_{H(v)})$-module introduced by Chen. It is clear that $\mathbf{N}_{v}$
is a faithful simple $L_{K}(G)$-module. Consider $\mathbf{N}_{v}$ as a simple
$L_{K}(E)$-module through the quotient map $L_{K}(E)\rightarrow L_{K}(G)$.
This simple module is denoted by $\mathbf{N}_{v}^{H(v)}$.

(iii) For any infinite path $p$, $V_{[p]}^{f}$ is the twisted simple
$L_{K}(E)$-module obtained from the simple $L_{K}(E)$-module $V_{[p]}$. See
\cite{AR1} for details.

\begin{proposition}
\label{Breaking vertex}If $v$ is an infinite emitter such that $s^{-1}(v)\cap
r^{-1}(M(v))|\neq0$ and is finite, then $\mathbf{S}_{v\infty}\cong%
$\ $\mathbf{N}_{v}^{B_{H(v)}}$.
\end{proposition}

\begin{proof}
From Proposition \ref{Annihilator Ideal} above and Lemma 3.5 of \cite{AR1}, it
is clear that both $\mathbf{S}_{v\infty}$ and\ $\mathbf{N}_{v}^{B_{H(v)}}$are
annihilated by the \ same primitive ideal. Also the $K$-bases of
$\mathbf{S}_{v\infty}$ and\ $\mathbf{N}_{v}^{B_{H(v)}}$ are in bijective
correspondence. Indeed if $p=p^{\prime}e$ is a path with $r(p)=v$, then, in
the graph $F$ defined in type-3 (i) simple module above, $r(e^{\prime
})=v^{\prime}$ and $s(e^{\prime})=s(e)=r(p^{\prime})$ and so $p^{\prime
}e^{\prime}$ is a path in $F$ with $r(p^{\prime}e^{\prime})=v^{\prime}$. Then
$v\longmapsto v^{\prime}$ and $p=p^{\prime}e\longmapsto p^{\prime}e^{\prime}$
is the desired bijection. It is then clear that the map $\phi:\mathbf{S}%
_{v\infty}\longrightarrow\mathbf{N}_{v}^{B_{H(v)}}$ given by $\phi
(v)=v^{\prime}$and $\phi(p^{\prime}e)=p^{\prime}e^{\prime}$ extends to an
isomorphism from $\mathbf{S}_{v\infty}$ to $\mathbf{N}_{v}^{B_{H(v)}}$.
\end{proof}

\begin{proposition}
\label{Empty intersection with M(v)}If $v$ is an infinite emitter for which
$|s^{-1}(v)\cap r^{-1}(M(v))|=0$, then $\mathbf{S}_{v\infty}\cong%
$\ $\mathbf{N}_{v}^{H(v)}$.
\end{proposition}

\begin{proof}
This is immediate after observing that these two simple modules have the same
$K$-basis and the same annihilating primitive ideal.
\end{proof}

\begin{notation}
In conformity with the notation used in \cite{AR1}, when $v$ is an infinite
emitter for which $|s^{-1}(v)\cap r^{-1}(M(v))|$ is infinite, we shall denote
the corresponding simple module $\mathbf{S}_{v\infty}$ by $\mathbf{N}%
_{v\infty}$.
\end{notation}

\begin{proposition}
\label{Different simple}The new simple module $\mathbf{N}_{v\infty}$ is not
isomorphic to any of the previously defined simple $L$-modules of Type 1, 2 or 3.
\end{proposition}

\begin{proof}
For convenience, we list the simple modules of type 1, 2 and 3 as
$\mathbf{N}_{w},\mathbf{N}_{v_{1}}^{B_{H(v_{1})}},\mathbf{N}_{v_{2}}%
^{H(v_{2})},V_{[p]}^{f},V_{[p]}$. Now $\mathbf{N}_{v\infty}\ncong V_{[p]}^{f}$
since the annihilator of $V_{[p]}^{f}$ is a non-graded primitive ideal
(\cite{AR1}, Lemma 2.4) while, as we proved in Proposition
\ref{Annihilator Ideal}, $Ann_{L}(\mathbf{S}_{v\infty})=I_{(H(v),B_{H(v)})}$
is a graded ideal. The proof that $\mathbf{N}_{v\infty}\ncong V_{[p]}$ uses
the same argument of Chen (\cite{C}, Theorem 3.7 (3)). We give the proof for
completeness. Suppose $\varphi:\mathbf{N}_{v\infty}\longrightarrow
\mathbf{V}_{[p]}$ is an $L$-morphism. We claim $\varphi=0$, that is,
$\varphi(v)=0$. Otherwise, write $\varphi(v)=%
{\textstyle\sum\limits_{i=1}^{n}}
k_{i}q_{i}$ where $q_{i}\in\lbrack p]$ and assume that the $q_{i}$ are all
different. Choose $n$ so that $\tau_{\leq n}(q_{i})$ are all pairwise
different. Now in the definition of $\mathbf{N}_{v\infty}$ as an $L$-module,
$e^{\ast}\cdot v=0$ for all $e\in E^{1}$ and so $\tau_{\leq n}(q_{1})^{\ast
}\cdot v$ $=0$, but $\varphi(\tau_{\leq n}(q_{1})^{\ast}\cdot v)=\tau_{\leq
n}(q_{1})^{\ast}\cdot\varphi(v)=k_{1}\tau_{>n}(q_{1})\neq0$, a contradiction.
Hence $\mathbf{N}_{v\infty}\ncong V_{[p]}$.

Since the annihilators of $\mathbf{N}_{w},\mathbf{N}_{v_{1}}^{B_{H(v_{1})}%
},\mathbf{N}_{v_{2}}^{H(v_{2})}$ and $\mathbf{N}_{v\infty}$ are all graded
ideals, it is enough if we can show that the set of vertices belonging to the
annihilators of these modules are all different. We first show that
$\mathbf{N}_{v\infty}\ncong\mathbf{N}_{w}$. Now the vertex set $H(w)\neq
H(v)$, since otherwise $M(w)=M(v)$ and this is not possible since $M(w)$
contains a sink (namely, $w$), while $M(v)$ does not. Hence $\mathbf{N}%
_{v\infty}\ncong\mathbf{N}_{w}$. Likewise, $H(v_{2})\neq H(v)$, since
otherwise $M(v_{2})=M(v)$ which will imply that $v_{2}\geq v$ in $M(v_{2})$
contradicting the fact that $v_{2}$ is a sink in $M(v_{2})$. So $\mathbf{N}%
_{v\infty}\ncong\mathbf{N}_{v_{2}}^{H(v_{2})}$. Finally, the annihilators of
$\mathbf{N}_{v_{1}}^{B_{H(v_{1})}}$and $\mathbf{N}_{v\infty}$ (being
$I_{(H,B_{H}\backslash\{v_{1}\})}$ and $I_{(H,B_{H})}$ respectively) are
different and so $\mathbf{N}_{v_{1}}^{B_{H(v_{1})}}$ $\ncong\mathbf{N}%
_{v\infty}$.\bigskip
\end{proof}

\section{The cardinality of the set of simple $L_{K}(E)$-modules}

As before, $E$ denotes an arbitrary graph with no restrictions on the
cardinality of $E^{0}$ or $E^{1}$. We wish to estimate the size of the
isomorphism classes of simple left $L_{K}(E)$-modules. In this connection, we
follow the ideas of Rosenberg \cite{RO}. However, we need to modify his
arguments for the case of Leavitt path algebras which, among other
differences, do not always have multiplicative identities. We first show that,
given a fixed simple module $S$, the cardinality of the set of all maximal
left ideals $M$ of $L_{K}(E)$ such that $L_{K}(E)/M\cong S$ is at most the
cardinality of $L_{K}(E)$. Using a Boolean subring of idempotents induced by
the paths in $L_{K}(E)$, we obtain a lower bound for the cardinality of the
set of non-isomorphic simple $L_{K}(E)$-modules. In particular, if $L_{K}(E)$
is a countable dimensional simple algebra, then it will have either exactly
$1$ or at least $2^{\aleph_{0}}$ distinct isomorphism classes of simple modules.

As before\textit{,} we shall denote\textbf{\ }the Leavitt path algebra\textbf{
}$L_{K}(E)$\textbf{ }by\textbf{ }$L$\textbf{. }We begin with a simple
description of maximal left ideals of $L$.

\begin{lemma}
\label{VertexSimple}Suppose $M$ is a maximal left ideal of $L$. Then for any
idempotent $\epsilon\notin M$, $M\epsilon\subset M$ and $M$ can be written as
$M=N\oplus L(1-\epsilon)$ where $N=M\cap L\epsilon=M\epsilon$. Every simple
left $L$-module $S$ is isomorphic to $Lv/N$ for some $v\in E^{0}$ and some
maximal $L$-submodule $N$ of $Lv$.
\end{lemma}

\begin{proof}
Let $M$ be a maximal left ideal of $L$ and $\epsilon=\epsilon^{2}\in
L\backslash M$. If $x\in M\cap L\epsilon$ then $x=x\epsilon$ and so $M\cap
L\epsilon\subset M\epsilon$. By maximality, $M\cap L\epsilon=M\epsilon$, so
$M\epsilon\subset M$ for all idempotents $\epsilon$. Writing each $x\in M$ as
$x=x\epsilon+(x-x\epsilon)$, we obtain $M=M\epsilon\oplus M(1-\epsilon)\subset
M\epsilon\oplus L(1-\epsilon)$. By maximality, $M=N\oplus L(1-\epsilon)$ where
$N=M\epsilon=M\cap L\epsilon$ is a maximal $L$-submodule of $L\epsilon$.

Suppose $S$ is a simple left $L$-module, say $S=L/M$ for some maximal left
ideal of $L$. Since $L=%
{\textstyle\bigoplus\limits_{v\in E^{0}}}
Lv$ and $M\neq L$, there is a vertex $v\notin M$. By the preceding paragraph,
we can write $M=N\oplus L(1-v)$ where $N=M\cap Lv$. Then $S=[(Lv\oplus
L(1-v)]/[N\oplus L(1-v)]\cong Lv/N$.
\end{proof}

\begin{lemma}
\label{IsoSimple}Suppose $Lv/N$ is a simple left $L$-module with $v\in E^{0}$.
Then, for any vertex $u$, a simple module $Lu/N^{\prime}$ is isomorphic to
$Lv/N$ if and only if there is an element $a=uav\in Lv$ such that $a\notin N$
and $N^{\prime}a\subset N$. In this case, $N^{\prime}=\{y\in Lu:ya\in N\}$.
\end{lemma}

\begin{proof}
Suppose $\sigma:Lu/N^{\prime}\rightarrow Lv/N$ is an isomorphism. Let
$\sigma(u+N^{\prime})=x+N$ for some $x\in Lv$. Now $\sigma(u+M)=\sigma
(u(u+M))=u(x+N)=ux+N$. Then $a=ux$ satisfies $a=uav$, $a\notin N$ and
$\sigma(u+N^{\prime})=a+N$. Moreover, $N^{\prime}a\subset N$ because, for any
$y\in N^{\prime}$, we have $ya+N=y(a+N)=y\sigma(u+M)=\sigma(y+M)=\sigma
(0+M)=0+N$. Note that the left ideal $I=\{y\in Lu:ya\in N\}$ contains
$N^{\prime}$ and $I\neq Lu$ since $u\notin I$ (as $ua=a\notin N$). Hence
$I=N^{\prime}$, by the maximality of $N^{\prime}$.

Conversely, suppose $N^{\prime}a\subset N$ for some $a$ satisfying $a\notin N$
and $a=uav$. Define $f:Lu/N^{\prime}\rightarrow Lv/N$ by $f(y+N^{\prime
})=ya+N$. Now $N^{\prime}a\subset N$ implies that $f$ is well-defined and is a
homomorphism. Now $f\neq0$ since $f(u+N^{\prime})=ua+N=a+N\neq N$. As both
$Lu/N^{\prime}$ and $Lv/N$ are simple modules, $f$\ is an isomorphism.
\end{proof}

\begin{lemma}
\label{DistinctSimple}Let $v$ be a vertex and $A=Lv/N$ be a simple left
$L$-module. Suppose, for $u,w\in E^{0}$, $B=Lu/N_{1}$ and $C=Lw/N_{2}$ are
both isomorphic to $A$ and $b=ubv$ and $c=wcv$ are the corresponding elements
satisfying \ $b,c\notin N$, $N_{1}b\subset N$ and $N_{2}c\subset N$ as
established in Lemma \ref{IsoSimple}. Then $B\neq C$ implies $b\neq c$.
\end{lemma}

\begin{proof}
Suppose, on the contrary, $b=c$. First of all $u=w$ since otherwise
$b=ub=uc=uwcv=0$, a contradiction. Thus $u=w$ and $N_{1},N_{2}$ are maximal
submodules of $Lu$. Then $N_{1}b\subset N$ and $N_{2}b\subset N$ implies
$(N_{1}+N_{2})b=Lub\subset N$. Since $b=ub$, we get $b\in N$, a contradiction.
\end{proof}

From the preceding Lemmas we get the following Proposition.

\begin{proposition}
\label{Card Iso Simple}(a) Let $Lv/N$ be a given simple left $L$-module, where
$v\in E^{0}$. For any fixed vertex $u$, the cardinality of the set of all
maximal submodules $N^{\prime}$ of $Lu$ (and thus the cardinality of all
maximal left ideals of $L$ of the form $M=N^{\prime}\oplus L(1-u)$ of $L$) for
which $Lu/N^{\prime}\cong Lv/N\cong L/M$ is at most the cardinality of $uLv$.

(b) Given a fixed simple left $L$-module $Lv/N$, the cardinality of the set of
maximal left ideals $M$ of $L$ for which $L/M\cong Lv/N$ is at most the
cardinality of $L$.
\end{proposition}

For subsequent applications, we obtain a sharpened version of Lemma
\ref{IsoSimple} as follows.

\begin{lemma}
\label{Annihilator} Let $v\in E^{0}$and $Lv/N$ be a simple left $L$-module.
Then, the maximal left ideals $M$ of $L$ for which $L/M\cong Lv/N$ are
precisely the annihilators in $L$ of non-zero elements $a+N$ of $Lv/N$ with
$a=uav$ for some vertex $u$.
\end{lemma}

\begin{proof}
Suppose $L/M\cong Lv/N$ for some maximal left ideal $M$ of $L$. By Lemma
\ref{VertexSimple}, we can write $M=N^{\prime}\oplus L(1-u)$ where $u$ is a
vertex, $u\notin M$ and $N^{\prime}=M\cap Lu$. By Lemma \ref{IsoSimple}, there
is an element $a=uav\notin N$ so that $a+N$ is non-zero and $N^{\prime}=$
$\{y\in Lu:ya\in N\}=\{y\in Lu:y(a+N)=N\}$. It is then clear that $M=\{L\in
L:L(a+N)=N\}.$

Conversely, suppose the left ideal $I$ is the annihilator in $L$ of some
non-zero element $a+N$ of $Lv/N$, where $a\in uLv$ for some vertex $u$. Let
$N^{\prime}=\{y\in Lu:y(a+N)=N\}=$ $\{y\in Lu:ya\in N\}$. Now $N^{\prime}\neq
Lu$ since $u\notin N^{\prime}$ due the fact that $ua=a\notin N$. Define
$\phi:Lu/N^{\prime}\rightarrow Lv/N$ by $\phi(ru+N^{\prime})=rua+N$. Clearly
$\phi$ is a well-defined homomorphism and $\phi\neq0$, as $\phi(u)=ua+N=a+N$.
If $rua+N=N$, then $ru(a+N)=N$, so $ru\in N^{\prime}$ and $ru+N^{\prime
}=N^{\prime}$. Thus $\ker(\phi)=0$. Since $Lv/N$ is simple, $\phi
:Lu/N^{\prime}\rightarrow Lv/N$ is an isomorphism. In particular, $N^{\prime}$
is a maximal $L$-submodule of $Lu$. Then $M=N^{\prime}\oplus L(1-u)$ is a
maximal left ideal of $L$, $L/M\cong Lv/N$ and $M\subset I$. By maximality,
$M=I$, the annihilator of $a+N$ in $L$.
\end{proof}

In the context of Proposition \ref{Card Iso Simple}(b), our next goal is to
investigate the size of the set of all non-isomorphic simple left $L$-modules.
Towards this end, we consider maximal left ideals of $L$ that arise from a
specified Boolean subring of idempotents in $L$.

\textbf{A special Boolean subring }$B$ \textbf{of }$L$ : Let $S=E^{0}%
\cup\{\alpha\alpha^{\ast}:\alpha$ a finite path in $E\}\cup\{0\}$. Observe
that elements of $S$ are commuting idempotents. Moreover, if $a,b\in S$, then
it is easy to see that $ab\in S$. Let $B$ be the additive subgroup of $L$
generated by $S$. Define, for any two elements $a,b\in S$, $a\vartriangle
b=a+b-2ab$ and $a\cdot b=ab$. Then $B$ becomes a Boolean ring under the
operations $\vartriangle$ and $\cdot$.

Define a partial order $\leq$ on $B$ by setting, for any two elements $a,b\in
B$, $a\leq b$ if $ab=a$. Then $B$ becomes a lattice under the operations,
$a\vee b=a+b-ab$ and $a\wedge b=ab$.

\begin{proposition}
\label{Equal max ideals}(a) If $M^{\prime}$ is a maximal left ideal of $L$,
then $M=M^{\prime}\cap B$ is a maximal ideal of $B$ and $M^{\prime}=N\oplus
L(1-v)$ for some vertex $v\notin M^{\prime}$ where $N=M^{\prime}\cap
Lv=M^{\prime}v$ and $M=Mv\oplus B(1-v)$.

(b) Every maximal ideal $M$ of $B$ embeds in a maximal left ideal $P_{M}$ of
$L$ such that $P_{M}\cap B=M$. Thus different maximal ideals $M_{1},M_{2}$ of
$B$ give rise to different maximal left ideals $P_{M_{1}},P_{M_{2}}$.
\end{proposition}

\begin{proof}
(a) If $M^{\prime}$ is a maximal left ideal of $L$, then clearly $M=M^{\prime
}\cap B$ is an ideal of $B$. To show that $M$ is maximal, it is enough if we
show that $M$ is a prime ideal of $B$. Suppose $x,y\in B$ such that $xy\in M$
and $x\notin M$. Since $Lx+M^{\prime}=L$, we can write $y=rx+m^{\prime}$ where
$r\in L$ and $m^{\prime}\in M^{\prime}$. Then $y=y^{2}=rxy+m^{\prime}y$. By
Lemma \ref{VertexSimple}, $m^{\prime}y\in M^{\prime}y\subset M^{\prime}$ and
so $y\in M^{\prime}\cap B=M$. Thus $M$ is a maximal ideal of $B$. Let $v$ be a
vertex with $v\notin M^{\prime}$. By Lemma \ref{VertexSimple}, $M^{\prime
}=N\oplus L(1-v)$ where $N=M^{\prime}v$. Note that $v\in B$ and $Mv\subset M$,
as $M$ is an ideal. Thus $M=Mv\oplus M(1-v)\subset Mv\oplus B(1-v)$. By
maximality, $M=Mv\oplus B(1-v)$.

(b) Let $M$ be a maximal ideal of $B$. Then there is at least one vertex
$v\notin M$. Because if $E^{0}\subset M$, then for every path $\alpha$ with,
say $s(\alpha)=u$, $\alpha\alpha^{\ast}=u\alpha\alpha^{\ast}\in M$, as $M$ is
an ideal of $B$. This implies $M=B$, a contradiction. We now claim that the
left ideal $LMv\neq Lv$. Suppose, by way of contradiction, assume that $v\in
LMv$, so that $v=%
{\textstyle\sum\limits_{i=1}^{k}}
r_{i}m_{i}v$ where $m_{i}\in M$ and $r_{i}\in L$. Observing that $m=m_{1}%
\vee\cdot\cdot\cdot\vee m_{k}$ belongs to the ideal $M$ and satisfies
$m_{i}m=m_{i}$ for all $i$, we get $mv=vm=%
{\textstyle\sum\limits_{i=1}^{k}}
r_{i}m_{i}mv=%
{\textstyle\sum\limits_{i=1}^{k}}
r_{i}m_{i}v=v$. This is not possible, since $mv\in M$ while $v\notin M$. Thus
$LMv$ is a proper $L$-submodule of $Lv$ and hence can be embedded in a maximal
$L$-submodule $N$ of $Lv$. Writing each element $x\in M$ as $x=xv+(x-xv)$ we
see that $M$ embeds in the maximal left ideal $P_{M}=N\oplus L(1-v)$. By the
maximality of $M$, it is clear that $P_{M}\cap B=M$. This implies that if
$M_{1}\neq M_{2}$ are maximal ideals of $B$ embedding, as above, in maximal
left ideals $P_{M_{1}}$ and $P_{M_{2}}$ of $L$, then $P_{M_{1}}\neq P_{M_{2}}$.
\end{proof}

\begin{corollary}
\label{card max ideals}The cardinality of the set of all maximal left ideals
of $L$ is at least the cardinality of the set of all maximal ideals of $B$.
\end{corollary}

For each maximal ideal $M$ of $B$, choose one maximal left ideal
$P_{M}=N\oplus L(1-v)$ where $N=P_{M}v$ as constructed in Proposition
\ref{Equal max ideals}(b). Let $\boldsymbol{T}$ denote the set of all such
maximal left ideals $P_{M}$ of $L$. We shall call such $P_{M}$ a
\textit{Boolean maximal left ideal corresponding to the maximal ideal }%
$M$\textbf{ }\textit{of}\textbf{ }$B$ and call the simple module $L/P_{M}$ a
Boolean simple module.

From Proposition \ref{Equal max ideals} it is clear that for each vertex $v$
there is a Boolean maximal left ideal $P_{M}=P_{M}v\oplus L(1-v)$ not
containing $v$. Because, given $v$ we can find a maximal left ideal $Q$ of $L$
not containing $v$. Clearly $Q\cap B=M$ is a maximal ideal in $B$ not
containing $v$. Then proceed as on Proposition \ref{Equal max ideals}(b), to
construct the Boolean maximal left ideal $P_{M}$ corresponding to $M$ and, as
noted there, $P_{M}=P_{M}v\oplus L(1-v)$.

\begin{proposition}
\label{Cardinality Boolean simple} Let $Lv/N$ be a fixed simple left
$L$-module where $v\in E^{0}$ and $N$ is a maximal $L$-submodule of $Lv$. Let
$\mathbf{S}_{v,N}=\{P_{M}\in\mathbf{T}:L/P_{M}\cong Lv/N\}$. Let
$\sigma=|\mathbf{S}_{v,N}|$ and write $\mathbf{S}_{v,N}=\{P_{M_{\alpha}%
}=P_{M_{\alpha}}v_{\alpha}\oplus L(1-v_{\alpha}):v_{\alpha}\in E^{0}%
,\alpha<\sigma\}$.Then

(a) $|\mathbf{S}_{v,N}|\leq\dim_{K}(Lv/N)$;

(b) the cardinality of the set of all maximal left ideals $P$ of $L$ such that
$L/P\cong Lv/N$ is $\leq%
{\textstyle\sum\limits_{\alpha<\sigma}}
|v_{\alpha}Lv_{\alpha}|$.
\end{proposition}

\begin{proof}
(a) By Lemma \ref{Annihilator}, each $P_{M_{j}}\in\mathbf{S}_{v,N}$
annihilates an element $x_{j}=a_{j}+N\in Lv/N$. Regarding $Lv/N$ as a
$K$-vector space, we claim that these elements $x_{j}$ (corresponding to the
various $P_{M_{j}}\in\mathbf{S}_{v,N}$) must be $K$-independent. To justify
this, suppose a finite subset of the elements $x_{j}$, with $j=1,...,n+1$,
satisfy
\[%
{\textstyle\sum\limits_{j=1}^{n+1}}
k_{j}x_{j}=0............(\ast)
\]
where, for each $j$, $k_{j}\in K$ and the maximal ideal $P_{M_{j}}$ is the
corresponding annihilator of the element $x_{j}$. Observe that the maximal
ideals of $B$ satisfy the Chinese remainder theorem and so, corresponding to
the finite set $M_{1},\cdot\cdot\cdot,M_{n},M_{n+1}$ of maximal ideals of $B$,
there is an element $b\in\cap_{i=1}^{n}M_{i}$ such that $b\notin M_{n+1}$ so
that the ideal generated by $\cap_{i=1}^{n}M_{i}$ and $M_{n+1}$ is $B$. Since
the vertex set $E^{0}\subset B$, we then see that $(\cap_{j=1}^{n}P_{M_{j}%
})+P_{M_{n+1}}=L$ and so there is an element $a\in\cap_{j=1}^{n}P_{M_{j}}$,
but $a\notin P_{M_{n+1}}$. Since $a$ annihilates $x_{1},\cdot\cdot\cdot,x_{n}%
$, multiplying the equation $(\ast)$ on the left by the element $a$, we get
$k_{n+1}ax_{n+1}=0$ which implies $k_{n+1}=0$. Proceeding like this, we
establish the independence of the elements $x_{j}$. Thus the elements $x_{j}$
can be regarded as part of a basis of $Lv/N$. Since distinct maximal left
ideals $P_{M_{j}}$ correspond to different such elements $x_{j}$ in a basis of
$LvN$ (Lemma \ref{DistinctSimple}), we conclude that $|\mathbf{S}_{v,N}%
|\leq\dim_{K}(Lv/N)$.
\end{proof}

(b) Now, for a fixed $\alpha$, Proposition \ref{Card Iso Simple}(a) implies
that the cardinality of the set of all the maximal left ideals $P$ with $P\cap
B=P_{M_{\alpha}}\cap B$ (so $P=Pv_{\alpha}\oplus L(1-v_{\alpha})$) and
satisfying $L/P\cong L/P_{M_{\alpha}}$ ($\cong Lv/N$) is $\leq|v_{\alpha
}Lv_{\alpha}|$. So the cardinality of the set of all maximal left ideals $P$
such that $L/P$ is isomorphic to $Lv/N$ is $\leq%
{\textstyle\sum\limits_{\alpha<\sigma}}
|v_{\alpha}Lv_{\alpha}|$.

\begin{lemma}
\label{No min Element}If $Soc(L)=0$ and the graph $E$ satisfies Condition (L),
then the Boolean ring $B$ is atomless, that is, it has no minimal elements.
\end{lemma}

\begin{proof}
Since $Soc(L)=0$, the graph $E$ cannot have any line points and, in
particular, has no sinks. Suppose, by way of contradiction, $B$ has a minimal
element $m$ so that, for all $b\in B$, either $mb=0$ or $mb=m$. In order to
reach a contradiction, we first claim that $m$ can be taken to be a monomial
of the form $\gamma\gamma^{\ast}$ for some path $\gamma$. To see this, if
$m=v$ is a vertex, then as $v$ is not a sink, it will be the source of some
path $\alpha$ and in that case $m=m\alpha\alpha^{\ast}=v\alpha\alpha^{\ast
}=\alpha\alpha^{\ast}$. Likewise, suppose $m=%
{\textstyle\sum\limits_{i=1}^{k}}
t_{i}\alpha_{i}\alpha_{i}^{\ast}$ where $t_{i}\in K$ and $\alpha_{i}\alpha
_{i}^{\ast}\neq\alpha_{j}\alpha_{j}^{\ast}$ for all $i,j$. Assume, without
loss of generality that $\alpha_{1}$ is of maximal length. Then, observing
that $\alpha_{1}\alpha_{1}^{\ast}\alpha_{i}\alpha_{i}^{\ast}\neq0$ implies
that $\alpha_{1}\alpha_{1}^{\ast}\alpha_{i}\alpha_{i}^{\ast}=\alpha_{1}%
\alpha_{1}^{\ast}$, we conclude that $m=\alpha_{1}\alpha_{1}^{\ast}m=%
{\textstyle\sum\limits_{j=1}^{r}}
\alpha_{1}\alpha_{1}^{\ast}=r\alpha_{1}\alpha_{1}^{\ast}$. Since $m^{2}=m$,
$\ r=1$ and we conclude that $m=\alpha_{1}\alpha_{1}^{\ast}$. We thus conclude
that $m=\gamma\gamma^{\ast}$ for some path $\gamma$ with $s(\gamma)=u$.
Suppose $r(\gamma)=w$ ($w$ may be equal to $u)$. Since $w$ cannot be a line
point, there is a vertex in $T(w)$ which is either a bifurcation vertex or is
the base of a cycle in $T(w)$. Since every cycle has an exit (due to Condition
(L)), $T(w)$ will always contain a bifurcation vertex $w^{\prime}$ with
$w^{\prime}=s(e)=s(f)$ for some edges $e\neq f$. Denoting a path from $w$ to
$w^{\prime}$ by $\delta$, we obtain, by the minimality of $m=\gamma
\gamma^{\ast}$,%
\[
\gamma\gamma^{\ast}=\gamma\gamma^{\ast}\gamma\delta ee^{\ast}\delta^{\ast
}\gamma^{\ast}.
\]
Multiplying on the right by $\gamma\delta f$, we get $\gamma\gamma^{\ast
}\gamma\delta f=\gamma\gamma^{\ast}\gamma\delta ee^{\ast}\delta^{\ast}%
\gamma^{\ast}\gamma\delta f$ \ and from this we get $\gamma\delta
f=\gamma\gamma^{\ast}\gamma\delta ee^{\ast}f=0$, a contradiction. This proves
that the Boolean ring $B$ has no minimal element.
\end{proof}

\begin{theorem}
\label{Countable dimesion}Let $E$ be an arbitrary graph satisfying Condition
(L). If $L=L_{K}(E)$ is a countable dimensional $K$-algebra with $Soc(L)=0$,
then $L$ has at least $2^{\aleph_{0}}$ distinct isomorphism classes of simple
left $L$-modules.
\end{theorem}

\begin{proof}
Consider the Boolean ring $B$ of $L$ defined earlier. By Lemma
\ref{No min Element}, the Boolean ring $B$ has no minimal elements. Thus $B$
is a countable atomless Boolean ring without identity. In this case, it is
well-known (see \cite{K} or Theorems 1, 8 and 13 in \cite{MHS}) that the space
$X$ of all maximal ideals of $B$ is a locally compact totally disconnected
Hausdorff space with no isolated points. Let $X^{\ast}$ be the one-point
compactification of $X$ obtained by the adjunction of a single non-isolated
point to $X$. Now $X^{\ast}$ is homeomorphic to the Cantor set (see
\cite{MHS}) and so $X^{\ast}$, and hence $X$, has cardinality $2^{\aleph_{0}}%
$. Thus $B$ has $2^{\aleph_{0}}$ distinct maximal ideals. From Corollary
\ref{card max ideals} we conclude that there is a set $\mathbf{T}$\textbf{ }of
$2^{\aleph_{0}}$ distinct Boolean maximal left ideals of $L$ obtained from the
ideals of $B$. Now for any given maximal ideal $P=N\oplus L(1-v)\in\mathbf{T}%
$, the set $S_{P}=\{Q\in\mathbf{T}:L/Q\cong L/P\}$ is countable since $L$ and
hence $Lv/N$ has countable $K$-dimension and, by Proposition
\ref{Cardinality Boolean simple}, $|S_{P}|\leq\dim_{K}(Lv/N)$. Since
$|\mathbf{T|=}2^{\aleph_{0}}$, $L$ has $2^{\aleph_{0}}$ non-isomorphic Boolean
simple $L$-modules of the form $L/P$ where $P\in\mathbf{T}$. Consequently, $L$
has at least $2^{\aleph_{0}}$ non-isomorphic simple left $L$-modules.
\end{proof}

\begin{corollary}
Let $E$ be an arbitrary graph. If $L=L_{K}(E)$ is a simple countable
dimensional $K$-algebra, then $L$ has either exactly $1$ or at least
$2^{\aleph_{0}}$ distinct isomorphism classes of simple left $L$-modules.
\end{corollary}

\begin{proof}
If $Soc(L)\neq0$, then, by simplicity, $L=Soc(L)$ and all the simple left
$L$-modules are isomorphic. Suppose $Soc(L)=0$, then the simplicity of $L$
implies that the graph $E$ satisfies Condition (L) (see \cite{AA}). We then
obtain the desired conclusion from Theorem \ref{Countable dimesion}.
\end{proof}

\begin{remark}
The method of proof of Theorem \ref{Countable dimesion} breaks down if $E$ is
an uncountable graph. Because, unlike the case of a countable atomless Boolean
ring, the set of maximal ideals of an uncountable atomless Boolean ring $B$
may not have the desired larger cardinality than $|B|$, unless some conditions
such as completeness of $B$ holds, or if $B$ has an independent subset of
cardinality $|B|$ (I am grateful to Professor Stefan Geschke for this remark.
See \cite{K} for details). When $E$ is an uncountable graph, the Boolean ring
$B$ that we constructed in the proof of Theorem \ref{Countable dimesion} need
not be complete and also need not have a large enough independent subset.
\end{remark}

\section{Finitely presented simple modules}

Let $E$ be a finite graph. It was shown in \cite{AR1} that every simple left
$L_{K}(E)$-module is finitely presented if and only if distinct cycles in $E$
are disjoint, that is, they have no common vertex. Interestingly, in
\cite{AAJZ-1}, this same condition on the graph $E$ is shown to be equivalent
to the algebra $L_{K}(E)$ having finite Gelfand-Kirillov dimension. Further,
Theorem 1 of \cite{AAJZ-2} shows that if the graph $E$ has the stated
property, then $L=L_{K}(E)$ is the union of a finite ascending chain of
ideals
\[
0\subset I_{0}\subset I_{1}\subset\cdot\cdot\cdot\subset I_{m}=L
\]
where $I_{0}$ is a direct sum of finitely many matrix rings $M_{n}(K)$ over
$K$ with $n\in%
\mathbb{N}
\cup\{\infty\}$ and, for $j\geq1$, each successive quotient $I_{j}/I_{j-1}$ is
a direct sum of finitely many matrix rings $M_{n}(K[x,x^{-1}])$ over
$K[x,x^{-1}]$ with $n\in%
\mathbb{N}
\cup\{\infty\}$. In this section, we show that the converse of the above
statement holds and obtain an improved version of the statement and proof of
Theorem 1 of \cite{AAJZ-2} (see Theorem \ref{No intersecting cycles in E}
below). An easy proof of Theorem 2 \ of \cite{AAJZ-2} is also pointed out.

We begin with the following easily derivable Lemma which was implicit in
\cite{AA} and was proved in \cite{CO}.

\begin{lemma}
\label{vertex on cycle}Let $E$ be any graph and let $H$ be a hereditary subset
of vertices in $E$. If $w$ is the base of a closed path and if $w\in\bar{H}$
the saturated closure of $H$, then $w\in H$.
\end{lemma}

In addition to proving the converse of Theorem 1 of \cite{AAJZ-2}, the next
theorem consolidates the various properties of the algebra $L_{K}(E)$ where
the graph $E$ has the mentioned property.

\begin{theorem}
\label{No intersecting cycles in E}Let $E$ be a finite graph and let $K$ be
any field. Then the following are equivalent for the Leavitt path algebra
$L=L_{K}(E)$:

(i) \ No two distinct cycles in $E$ have a common vertex;

(ii) \ Every simple left $L$-module is finitely presented;

(iii) $L$ has finite Gelfand-Kirillov dimension;

(iv) $L$ is the union of a finite ascending chain of graded ideals
\[
0\subset I_{0}\subset I_{1}\subset\cdot\cdot\cdot\subset I_{m}=L\qquad
\qquad\qquad\qquad(\ast)
\]
with $H_{j}=I_{j}\cap E^{0}$, where $I_{0}=Soc(L)$ and, for each $j\geq0$,
identifying, $L/I_{j}$ with $L_{K}(E\backslash H_{j})$, $I_{j+1}/I_{j}$ is the
ideal generated by the vertices in all the cycles without exits in
$E\backslash H_{j}$ and $Soc(L/I_{j})=0$.

(v) $L$ is the union of a finite ascending chain of graded ideals
\[
0\subset I_{0}\subset I_{1}\subset\cdot\cdot\cdot\subset I_{m}=L
\]
where $I_{0}$ is a direct sum of finitely many matrix rings of the form
$M_{n}(K)$ where $n\in%
\mathbb{N}
\cup\{\infty\}$ and for each $j>1$, $I_{j}/I_{j-1}$ is a direct sum of
finitely many matrix rings of the form $M_{n}(K[x,x^{-1}])$ where $n\in%
\mathbb{N}
\cup\{\infty\}$.

(vi) $E^{0}$ is the union of a finite ascending chain of hereditary saturated
subsets%
\[
H_{0}\subset\cdot\cdot\cdot\subset H_{m}=E^{0}%
\]
where $H_{0}$ is the hereditary saturated closure of all the line points in
$E$ and, for each $j\geq0$, $E\backslash H_{j}$ has no line points and
$H_{j+1}\backslash H_{j}$ is the hereditary saturated closure of the set of
vertices in all the cycles without exits in the graph $E\backslash H_{j}$.
\end{theorem}

\begin{proof}
The equivalence of (i) and (ii) was proved in \cite{AR1} and that of (i) and
(iii) was proved in \cite{AAJZ-1}.

Now (i) =%
$>$
(iv) follows from the proof of Theorem 1 of \cite{AAJZ-2}. We give a slightly
different streamlined proof. Let $I_{0}=Soc(L)$, so $I_{0}$ is the ideal
generated by all the line points in $E$ (\cite{AMMS}). Now, for any graded
ideal $J$ containing $Soc(L)$, $Soc(L/J)=0$. This is because, as the
hereditary saturated set $J\cap E^{0}=H$ contains all the line points in $E$,
the finiteness of $E$ implies that the quotient graph $E\backslash H$ contains
no sinks and hence no line points. Suppose for $n\geq0$ we have defined the
graded ideal $I_{n}\supseteqq I_{0}$. Let $H_{n}=I_{n}\cap E^{0}$. Then
$E\backslash H_{n}$ satisfies the same hypothesis as $E$ and has no sinks, so
that every vertex in it connects to a cycle. Moreover, we claim that
$E\backslash H_{n}$ contains cycles without exits. To see this, for any given
two cycles $c,c^{\prime}$ in $E\backslash H_{n}$, define $c\geq c^{\prime}$ if
there is a path from a vertex in $c$ to a vertex in $c^{\prime}$. Since no two
cycles in $E\backslash H_{n}$ have a common vertex, $\geq$ \ is antisymmetric
and hence a partial order. Clearly every cycle which is minimal in this
partial order has no exits in $E\backslash H_{n}$. Now $L/I_{n}\cong
L_{K}(E\backslash H_{n})$ and $Soc(L/I_{n})=0$. Define $I_{n+1}/I_{n}$ to be
the ideal generated by the vertices in all the cycles without exits in
$E\backslash H_{n}$. It is clear that $I_{n+1}$ is a graded ideal of $L$. By
induction on $n$, after a finite number of steps, we then obtain the chain
$(\ast)$ with the desired properties.

(iv) =%
$>$
(v) By Theorem 5.6 of \cite{AMMS}, $I_{0}$ is a direct sum of finitely many
matrix rings of the form $M_{n}(K)$ where $n\in%
\mathbb{N}
\cup\{\infty\}$ and, for each $j\geq0$, $I_{j+1}/I_{j}$ is, by Proposition 3.7
of \cite{AAPS}, a direct sum of finitely many matrix rings of the form
$M_{n}(K[x,x^{-1}])$ where $n\in%
\mathbb{N}
\cup\{\infty\}$.

(v) =%
$>$
(vi). \ Obvious from the proof of (iii) =%
$>$
(iv).

Assume (vi). Suppose, by way of contradiction, there is a vertex $w$ which is
the base of two distinct cycles $g,h$. Now $w\notin H_{0}$ since otherwise, by
Lemma \ref{vertex on cycle}, $w$ will be a line point in $E$, a contradiction.
Let $t\geq0$ be the smallest integer such that $w\notin H_{t}$, so $w\in
H_{t+1}\backslash H_{t}$. Now $H_{t+1}\backslash H_{t}$ is the saturated
closure of the set $S_{t}$ of all the vertices on cycles without exits in the
quotient graph $E\backslash H_{t}$. Then, by Lemma \ref{vertex on cycle},
$w\in S_{t}$, a contradiction. This proves (i).
\end{proof}

Following the ideas in \cite{AR1}, we illustrate Theorem
\ref{No intersecting cycles in E} by the simplest example of the Toeplitz algebra.

\begin{example}
Let $E$ be the graph with two vertices $v,w$, an edge $f$ with $s(f)=v,r(f)=w$
and a loop $c$ with $s(c)=v=r(c)$. Since $w$ is the only line point, the socle
of $L=L_{K}(E)$ is $S=<w>$ (see \cite{AMMS}) and there is an epimorphism
$L\longrightarrow K[x,x^{-1}]$ with kernel $S$ mapping $v$ to $1$, $c$ to $x$
and $c^{\ast}$ to $x^{-1}$. Thus $L/S\cong K[x,x^{-1}]$ and, moreover,
$S=Lw\oplus%
{\textstyle\bigoplus\limits_{n=0}^{\infty}}
Lwf^{\ast}(c^{\ast})^{n}$ is the direct sum of simple left ideals in $L$. We
wish to show that every simple left $L$-module $A=L/M$ is cyclically (hence
finitely) presented, where $M$ is a maximal left ideal of $L$.

If $S\nsubseteqq M$, then we have a direct decomposition $S=(S\cap M)\oplus T$
and $L=M\oplus T$. If $1=\epsilon+\epsilon^{\prime}$ with $\epsilon$ in $M$
and $\epsilon^{\prime}\in T$, then $M=L\epsilon$ is cyclic.

Suppose $S\subseteqq M$. Then there is an irreducible polynomial
$p(x)=1+k_{1}x+\cdot\cdot\cdot+k_{m}x^{m}\in K[x,x^{-1}]$ such that
$M/S=<p(x)>$. So $M=Lp(c)+S=Lp(c)+Lw+%
{\textstyle\bigoplus\limits_{n=0}^{\infty}}
Lwf^{\ast}(c^{\ast})^{n}=Lp(c)+%
{\textstyle\bigoplus\limits_{n=0}^{\infty}}
Lf^{\ast}(c^{\ast})^{n}$ as $wp(c)=w$. Let $N=%
{\textstyle\bigoplus\limits_{i=0}^{m-1}}
Lf^{\ast}(c^{\ast})^{i}$. Suppose $r\geq m-1$ and that $f^{\ast}(c^{\ast}%
)^{t}\in Lp(c)+N$ for all $t\leq r$. Then, $f^{\ast}(c^{\ast})^{r+1}=f^{\ast
}(c^{\ast})^{r+1}p(c)-k_{1}f^{\ast}(c^{\ast})^{r}-\cdot\cdot\cdot-k_{m}%
f^{\ast}(c^{\ast})^{r+1-m}\in Lp(c)+N$. Thus we conclude that
$Lp(c)+S=Lp(c)+N$. Observing that $\{f^{\ast}(c^{\ast})^{i}:i=0,\cdot
\cdot\cdot,m-1\}$is a set of mutually orthogonal elements, we get $N=%
{\textstyle\bigoplus\limits_{i=0}^{n-1}}
Lf^{\ast}(c^{\ast})^{i}=Lb$ where $b=f^{\ast}+f^{\ast}c^{\ast}+\cdot\cdot
\cdot+f^{\ast}(c^{\ast})^{n-1}$. Further, $p(c)f^{\ast}(c^{\ast})^{i}=0$ for
all $i$ and that $p(c)=p(c)v\in Lv$. Consequently, $M=Lp(c)+S=L(v+b)$ is
cyclic. This proves that the simple module $L/M$ is cyclically presented.
\end{example}

REMARK: Observe that, in our proof above, we never used the fact that the
polynomial $p(x)$ is irreducible. Since $K[x,x^{-1}]$ is a principal ideal
domain, the same argument shows that every left ideal $A\varsupsetneqq S$ in
$L$ is a principal left ideal. Also, if $A$ is a left ideal such that
$S\nsubseteqq A$ and $A\nsubseteqq S$, then decomposing $S=(S\cap A)\oplus T$,
we see that the left ideal $A+S=A\oplus T\varsupsetneqq S$. Thus $A\oplus T$
and hence $A$ is a principal left ideal in $L$. On the other hand, if
$A\subseteqq S$, $A$ need not be a principal left \ ideal. This is clear if
$A=S$, as $S$ is a direct sum of infinitely many simple left ideals. In
particular, $S$ is not a direct summand of $L$. Thus we obtain an easy proof
of the following proposition which occurs as Theorem 2 in \cite{AAJZ-2}.

\begin{proposition}
Let $E$ be a graph with two vertices $v,w$ and two edges $c,f$ with
$s(c)=v=r(c),s(f)=v,r(f)=w$. If $S=<w>$ is the two-sided ideal generated by
$w$, then $S$ cannot be a direct summand of $L=L_{K}(E)$ as a left $L$-module.
\end{proposition}

\end{document}